\documentclass[a4paper,10pt]{article}
\usepackage[margin=1in]{geometry}  

\usepackage{graphicx}              
\usepackage{amsmath}
\usepackage{amscd}               
\usepackage{amsfonts} 
\usepackage{amssymb}             
\usepackage{amsthm}                

\newtheorem{thm}{Theorem}[section]
\newtheorem{defn}[thm]{Definition}
\newtheorem{lem}[thm]{Lemma}
\newtheorem{prop}[thm]{Proposition}
\newtheorem{cor}[thm]{Corollary}

\newtheorem{rmk}[thm]{Remark}

\newcommand{\RR}{\mathbb{R}}      
\newcommand{\ZZ}{\mathbb{Z}}      
\newcommand{\NN}{\mathbb{N}}

\newcommand{\cL}{\mathcal{L}}

\newcommand{\cF}{\mathcal{F}}    

\newcommand{\EE}{\mathbb{E}}     
\newcommand{\PP}{\mathbb{P}}     

\newcommand{\cov}{\text{cov }}
\newcommand{\var}{\text{var }}

\newcommand{\lfl}{\left\lfloor}
\newcommand{\rfl}{\right\rfloor} 
\begin{document}

\title{On number of turns in reduced random lattice paths}
\author{
Yunjiang Jiang\\
Department of Mathematics\\
Stanford University\\
jyj@math.stanford.edu\\
\and
Weijun Xu\\
Mathematical and Oxford-Man Institutes\\
University of Oxford\\
xu@maths.ox.ac.uk\\
}

\maketitle

\abstract{
We consider the tree-reduced path of symmetric random walk on $\ZZ^{d}$. It is interesting to ask about the number of turns $T_n$ in the reduced path after $n$ steps. This question arises from inverting the signatures of lattice paths. We show that, when $n$ is large, the mean and variance of $T_n$ in the asymptotic expansion have the same order as $n$, while the lower terms are $O(1)$. We also obtain limit theorems for $T_n$, including large deviations principle, central limit theorem, and invariance principle. Similar techniques apply to other finite patterns in a lattice path.}

\bigskip


\bigskip

\textbf{Mathematics Subject Classification:} 60

\bigskip

\section{Introduction}

Let $G$ be the free group with $d$ generators $e_1, \cdots, e_d$. Start with the empty word at time $0$. At each time $k$, choose one from the $2d$ elements ($d$ generators and their inverses) uniformly randomly to multiply the current word on the right. For example, the first six choices: 
\begin{align*}
e_2, e_3, {e_3}^{-1}, e_2, {e_1}^{-1}, e_4
\end{align*}
will produce the reduced word $e_{2}e_{2}{e_{1}}^{-1}{e_4}$ at time 6. Every word at time $n$ has a unique reduced word with length at most $n$. It is then interesting to ask about the length and number of turns in the reduced word. 

\bigskip

\begin{defn}
Let $w$ be a word, and $\hat{w}=a_{i_1} \cdots a_{i_k}$ be its reduced word, where $a_{j}$ is either $e_j$ or ${e_{j}}^{-1}$. Define the number of turns of $w$ to be $T_n = \# \{a_{i_j}a_{i_{j+1}}: a_{i_j} \neq a_{i_{j+1}}, 1 \leq j \leq k-1\}$. 
\end{defn}

\bigskip

Then, two words have the same number of turns if they reduce to the same word. In the above example, the number of turns in the reduced word $e_{2}e_{2}{e_{1}}^{-1}{e_4}$ is $2$. In another language, it is the number of times a reduced path of random walk has switched its direction. 

The main goal of this paper is to calculate asymptotics for $T_{n}$ when $n$ is large. The question of estimating $T_n$ arises from inverting signature for lattice paths, where at most $T+2$ terms in the signature are needed for inversion if one knows in advance that the reduced lattic path has $T$ turns.

\bigskip

\subsection{Motivation from inversion of signature for axis paths}

In this subsection, we give some background material on path-signature that motivates our study of the current problem. A path $\gamma: [s,t] \rightarrow \RR^{d}$ is a continuous function mapping a time interval into $\RR^{d}$. The length of the path is defined as
\begin{align*}
|\gamma|:= \sup_{\mathcal{P}} d(\gamma(t_{i}),\gamma(t_{i+1})), 
\end{align*}
where $d$ the metric on $\RR^{d}$, and the supremum is taken over all finite partitions of $[s,t]$. If $|\gamma| < +\infty$, we say $\gamma$ has bounded variation. Let $BV(\RR^{d})$ denote the space of all paths of bounded variations in $\RR^{d}$. 

\bigskip

\begin{defn}
Let $\gamma: [s,t] \rightarrow \RR^{d}$ be an element in $BV(\RR^{d})$. The signature of $\gamma$, $X_{s,t}(\gamma)$, is defined as: 
\begin{align*}
X_{s,t}(\gamma) = 1 + X_{s,t}^{1}(\gamma) + \cdots + X_{s,t}^{n}(\gamma) + \cdots,  
\end{align*}
where
\begin{align} \label{tensor product}
X_{s,t}^{n}(\gamma) = \int_{s<u_{1} \cdots < u_{n}<t}{d\gamma(u_1) \otimes \cdots \otimes d\gamma(u_{n})}
\end{align}
as an element in $(\RR^{d})^{\otimes n}$. 
\end{defn}

\bigskip

Let $(e_{1}, e_{2}, \cdots, e_{d})$ be a standard basis of $\RR^{d}$, then $\gamma$ can be written as $(\gamma_{1}, \gamma_{2}, \cdots, \gamma_{d})$. If $w = e_{i_1} \cdots e_{i_n}$ is a word of length $n$, we write
\begin{align*}
C_{s,t}(w)=C_{s,t}(e_{i_1} \cdots e_{i_n}) = \int_{s < u_{1} < \cdots < u_{n} <t}{d\gamma_{i_1}(u_1) \cdots d\gamma_{i_n}(u_n)}
\end{align*}
as the coefficient of $w$. As all words of length $n$ form a basis of $V^{\otimes n}$, we can rewrite $X_{s,t}^{n}(\gamma)$ as the linear combination of basis elements: 
\begin{align} \label{power series}
X_{s,t}^{n}(\gamma) = \sum_{|w|=n}C_{s,t}(w)w, 
\end{align}
where the sum is taken over all words of length $n$. 

Re-parametrizing the path does not change the signature. For any path $\alpha: [0,s] \rightarrow V$ and $\beta: [0,t] \rightarrow V$, we can form the concatenation $\alpha * \beta: [0,s+t] \rightarrow V$, as follows: 
\begin{equation*} 
\alpha*\beta(u) := \left \{
\begin{array}{rl}
&\alpha(u), u \in [0,s]\\
&\beta(u-s)+\alpha(s)-\alpha(0), u \in [s,s+t]
\end{array} \right., 
\end{equation*}
and similarly, the decomposition of one path into two can be carried out in the same fashion. 

For any path $\gamma: [s,t] \rightarrow V$, the path "$\gamma$ run backwards", $\gamma^{-1}$, is defined as: 
\begin{align*}
\gamma^{-1}(u) :=\gamma (s+t-u), u \in [s,t], 
\end{align*}
and the trajectories of $\gamma * \gamma^{-1}$ cancel out each other. 

Concatenation and "backwards" of paths of bounded variation are still paths of bounded variation. In fact, we have $|\alpha * \beta| = |\alpha| + |\beta|$, and $|\gamma^{-1}| = |\gamma|$. The following proposition, first proved by Chen (\cite{Chen}), asserts that the signature map is a homomorphism from $BV(\RR^{d})$ to the tensor algebra. 

\bigskip

\begin{prop} \label{Chen's identity}
Let $\alpha, \beta \in BV^(\RR^{d})$. Then, $X(\alpha * \beta) = X(\alpha) \otimes X(\beta)$. 
\end{prop}

\bigskip

Hambly and Lyons (\cite{Hambly and Lyons}) showed that if $\alpha, \beta \in BV(\RR^{d})$, then $X(\alpha) = X(\beta)$ if and only if $\alpha * \beta^{-1}$ is tree-like, a continuous analogue of a null path. This tree-like relation defines an equivalence relation on $BV(\RR^{d})$. Within every equivalent class, there is a unique path with minimal length, called the tree reduced path. An interesting question would be, given a signature $X$ of some path of bounded variation, can one reconstruct the tree reduced path with the same signature $X$? 

For the case of axis paths, the answer was provided by Lyons and Xu in \cite{Lyons and Xu}. 

\bigskip

\begin{defn}
$\gamma: [s,t] \rightarrow \RR^{d}$ is a (finite) axis path if its movements are parallel to the Euclidean coordinate axes, has finitely many turns, and each straight line component has finite length. 
\end{defn}

\bigskip

Any axis path has a unique reduced axis path; integer lattice paths are special cases of axis paths. An $\RR^{d}$ axis path can move in $d$ different directions (up to the sign). At time $0$, it starts to move along a direction $e_{i_1}$ for some distance $r_1$; then it turns a right angle, and moves along $e_{i_2}$ for a distance $r_2$, and so on, and stops after finitely many turns. Thus, up to re-parametrization, an axis path $\gamma$ can be represented as: 
\begin{align} \label{representation}
\gamma = (r_{1}e_{i_1}) *  \cdots * (r_{n}e_{i_n})
\end{align}
where $r_i$'s are real numbers, with the sign denoting the direction\protect\footnote{We mean $-r e_{j} = r e_{j}^{-1}$.}. 

Using Chen's identity (Proposition \ref{Chen's identity}), the signature of $\gamma$ can be conveniently expressed as
\begin{align*}
X(\gamma) = \exp(r_{1}e_{i_1}) \otimes \cdots \otimes \exp(r_{n}e_{i_n}), 
\end{align*}
which should be understood as the product of $n$ power series in the letters $\{e_{i_1}, \cdots, e_{i_n}\}$. 

If $\gamma$ is already in its reduced form, then it is clear that $i_{k} \neq i_{k+1}$, and we call the word $w = (e_{i_1}, \cdots, e_{i_n})$ the shape of $\gamma$. If a word $w$ is in its reduced form, we use $|w|$ to denote the number of letters in $w$, or the length of $w$. We introduce the notion of square free words to characterize an axis path. 

\bigskip

\begin{defn}
Let $w = e_{i_1} \cdots e_{i_n}$ be a word. We call it a square free word if $\forall k \leq n-1$, $i_{k} \neq i_{k+1}$. 
\end{defn}

\bigskip

The following theorem, provided by Lyons and Xu (\cite{Lyons and Xu}), gives an inversion procedure for finite axis paths. 

\bigskip

\begin{thm}
For any finite axis path $\gamma$, there exists a unique square free word $w$ with the property that $C(w) \neq 0$, and that if $w'$ is any other square free word with $C(w') \neq 0$, then $|w'| < |w|$. Furthermore, suppose the unique longest square free word is $w = e_{i_1} \cdots e_{i_n}$, and let
\begin{align*}
w_{k}:= e_{i_1} \cdots e_{i_{k-1}} e_{i_{k}}^{2} \cdot e_{i_{k+1}} \cdot e_{i_n}, 
\end{align*}
which has length $n+1$, then we have
\begin{align*}
\gamma = (r_{1}e_{i_1}) * \cdots * (r_{n}e_{i_n}), 
\end{align*}
where $r_k = \frac{2C(w_k)}{C(w)}$. 
\end{thm}

\bigskip

Thus, one sees  if an axis path has $n$ turns, then at most $n+2$ terms in the signature are needed for inversion. For a lattice path with length $L$, it can have at most $L-1$ turns, so we only need the first $L+1$ terms in the signature to recover it.

In pratice, lattice paths are often generated by drawing $n$ letters and their inverses uniformly randomly from an alphebet, and putting them in the order they are drawn. It is then intersting to ask about the number of turns in its reduced path.

\bigskip

\subsection{Outline of the method and summary of results}

If one writes $T_n = \sum_{i=1}^{n}{V_i}$, where $V_i$ denote the number of turns created at step $i$. In general, $V_i$ can be $1$, $0$, or $-1$, and are correlated. The distribution of $V_i$ depends on the whole history in the past. 

On the other hand, one can condition on the length of reduced path $L_n$. Then $T_n|L_n$ has a binomial distribution. A detailed study of $L_n$ yields asymptotic behaviors of $T_n$. A natural coupling $L_n = S_n + D_n$ simplifies the study of $L_n$, where $S_n$ is a sum of $n$ i.i.d.'s, and $D_n$ is dominated by a geometric random variable. \\

The main results in this paper are: 

\bigskip

\begin{flushleft}
\textbf{Proposition 2.5.} $\EE T_n - \frac {2(d-1)^2}{d(2d-1)}n \rightarrow -\frac{2d^{2}-4d+1}{d(2d-1)}$, $\var{T_n}-\frac{2(d-1)^{2}(5d-2)}{d^{2}(2d-1)^{2}}n$ \textit{also converges}. 
\end{flushleft}

\bigskip

\begin{flushleft}
\textbf{Theorem 3.1. (Large Deviations Principle)} \textit{The sequence of the laws for the random variables $\{\frac{T_n}{n}\}_{n \geq 1}$ satisfy the large deviations principle with rate function
\begin{align*}
I(x) = \sup_{\theta} [\theta x - \log h(\theta)], 
\end{align*}
where
\begin{align*} 
h(\theta) = \left \{
\begin{array}{rl}
& \frac{1}{2d} [2(d-1)e^{\theta} + \frac{2d-1}{1 + 2(d-1)e^{\theta}} + 1], \theta \geq \log \frac{\sqrt{2d-1}-1}{2(d-1)} \\
&\frac{\sqrt{2d-1}}{d}, \qquad \qquad \qquad \qquad \qquad                 \theta < \log \frac{\sqrt{2d-1}-1}{2(d-1)}
\end{array} \right., 
\end{align*}}
\end{flushleft}

\bigskip

\begin{flushleft}
\textbf{Theorem 4.5. (Invariance Principle)} \textit{For each $n$,define a $C^0([0,1])$-valued random variable $\{W^{(n)}_t: t\in [0,1]\}$ by }
\begin{align*}
W^{(n)}_t = \frac{1}{\sigma \sqrt{n}}[T_{tn}- \frac{2(d-1)^2}{d(2d-1)} tn]
\end{align*}
\textit{for $t \in \frac{1}{n} [n]$, and linearly interpolated for other values of $t$. Then the sequence converges in law to the standard one dimensional Brownian motion on $[0,1]$ as $n \to \infty$.}
\end{flushleft}

\bigskip

As a generalization, analogous results hold for the number of occurrences of any finite collection of finite length pattern $\mathcal{P} = \{P_i=(e_{i_1}, \ldots, e_{i_{k_i}} )\}$ in a lattice path: the key is to establish central limit theorem for the number of occurences of elements in $\mathcal{P}$ conditioned on the length of the path $L_t$, which is essentially an i.i.d. sum of $m$-dependent random variables (see \cite{Orey}), where $m$ is bounded above by $\max_i k_i$. The number of turns $T_t$ corresponds to $\mathcal{P} = \{P_{ij} = (e_i, e_j): i \neq j\}$. For the sake of clarity, we will focus only on the number of turns.

The paper is organized as follows: 

In section 2, we show that the lower order terms in the mean and variance of $L_n$ are $O(1)$; we then obtain similar results for $T_n$. A key ingradient in the derivation is to prove $\cov(S_n,D_n)$ is $O(1)$. We compare it with $\cov(S_{n+1},D_{n+1})$, and show that their differnce decays exponentially with $n$, thus proving convergence. 

Section 3 is devoted to the proof of the large deviations principle for $\{\frac{T_n}{n}\}$. We derive the rate function, and thus prove the principle, by comparison of the Laplace transform of $L_n$ with that of $S_n$. It turns out that the rate function deviates from the normal one as predicted by $S_n$ on the lower side of the real line. 

In section 4, we prove the central limit theorem and invariance principle for $T_n$. This result shows that, although the components of $T_n$ are correlated, the increments are still asymptotically independent under proper scaling.

\bigskip

\bigskip

\section{Lower order terms in mean and variance}

Let $L_n$ denote the length of the reduced path after $n$ steps. Then, $T_{n}|L_{n}$ has a binomial distribution with parameters $(L_{n}-1, \frac{2d-2}{2d-1})$. Let $X_i$ be a sequence of indenpedent and indentically distributed random variables with $\PP(X_i=1)=\frac{2d-1}{2d}$, and $\PP(X_i=-1)=\frac{1}{2d}$. Let $L_{0}=0$, then $L_n$ can be defined inductively as follows: 
\begin{align*}
L_{i+1}=L_{i}+X_{i+1}, L_{i}>0\\ 
L_{i+1}=1, L_{i}=0
\end{align*}
We want to compare $L_{n}-S_{n}$, where $S_{n}=\sum_{i=1}^{n}{X_i}$. It is well known that $\frac{L_n}{n} \rightarrow \frac{d-1}{d}$ almost surely. We compute a finer estimate to show that $\EE L_n - \frac{d-1}{d}n = O(1)$. 

Since $L_i - S_i$ does not change when $L$ is away from $0$, so difference only occurs when $L$ hits $0$. Let $R_n$ denote the number of times that $L_n$ hits to $0$ after the first step up to time $n$. Since $\PP$($L$ ever comes back to $0$)=$\frac{1}{2d-1}$, $R_n$ convereges to a geometric distributed random varible $R$, with $\PP (R=k) = \frac{2d-2}{2d-1}(\frac{1}{2d-1})^{k-1}$ (see e.g. \cite{Feller}). Here, step $0$ is counted as a return, because $L$ and $S$ can be different in the first move. Let $D_n = L_{n}-S_{n}$, then $\frac{1}{2}D_{n}|R_{n-1}$ has a binomial distribution with parameters $(R_{n-1}, \frac{1}{2d})$. In particular, $D_n \leq 2R_{n-1}$. So the mean of the difference is: 
\begin{align*}
\EE D_n = \EE \EE (D_n | R_{n-1}) = \frac{1}{d} \EE R_{n-1} \rightarrow \frac{2d-1}{2d(d-1)}
\end{align*}

Thus, we get an error term for $\EE{L_n}$: 

\bigskip

\begin{lem}
Let $L_n$ denote the length of the reduced word after $n$ steps, then
\begin{align*}
\lim_{n \rightarrow \infty}(\EE L_n - \frac{d-1}{d}n)=\frac{2d-1}{2d(d-1)}. \\
\end{align*}
\end{lem}

\bigskip

Now to compute the lower order terms in $\var L_n$. Since  $\var L_n = \var{S_n} + 2 \cov(S_n, D_n) + \var{D_n}$, it suffices to show that $\cov(S_n, D_n) = O(1)$. We show it in the following lemma. 

\bigskip

\begin{lem}
Under the above coupling $L_n = S_n + D_n$, we have: 
\begin{align*}
\lim_{n \rightarrow \infty}(\EE S_n D_n - \EE S_n \EE D_n)= -u(d), 
\end{align*}
where $0 \leq u(d) \leq \frac{2d^{2}(2d-1)(d^{2}+2d-1)}{(d-1)^{5}}$. 
\end{lem}

\begin{proof}
Let $U_n = (S_{n} - \frac{d-1}{d}n)D_{n}$. We show that $\EE U_{n+1} - \EE U_n$ decays exponentially fast. 
\begin{align*}
U_{n+1}&=1_{\{L_{n}=0\}}(S_{n}+X_{n+1}-\frac{d-1}{d}(n+1))(D_{n}+1-X_{n+1})+1_{\{L_{n}>0\}}(S_{n}+X_{n+1}-\frac{d-1}{d}(n+1))D_{n}\\
&=U_{n}+(X_{n+1}-\frac{d-1}{d})D_{n}+1_{\{L_{n}=0\}}(1-X_{n+1})(S_{n+1}-\frac{d-1}{d}(n+1))
\end{align*}

Since $X_{n+1}$ is independent of $D_n$, and $-n \leq S_n \leq 0$ when conditioned on $L_n=0$, taking expectation on both sides yields: 
\begin{align*}
-4(n+1) \PP (L_n=0) \leq \EE U_{n+1} - \EE U_{n} \leq 0. 
\end{align*}

Since $\PP(L_{2n+1}=0)=0$, and
\begin{align*}
\PP(L_{2n}=0) & \leq \PP(S_{2n} \leq 0)\\
& \leq \sum_{k=0}^{n}{2^{2n}}(\frac{1}{2d})^{n+k}(\frac{2d-1}{2d})^{n-k}\\
& \leq \frac{2d-1}{2(d-1)}(\frac{2d-1}{d^2})^{n}
\end{align*}

This shows that $\EE U_n$ is decreasing and bounded below, and thus it has a finite limit. 

Adding up all $(\EE U_n - \EE U_{n-1})$ gives $\EE U_n \rightarrow -u(d)$, where $0 \leq u(d) \leq \frac{2d^{2}(2d-1)(d^{2}+2d-1)}{(d-1)^{5}}$. 
\end{proof}

\bigskip

\begin{rmk}
The negative correlation agrees with one's probabilistic intuition: when $S_n$ is small, the process $L_n$ tends to visit $0$ more times, and thus $D_n$ is likely to be large. 
\end{rmk}

\bigskip

\begin{prop}
Let $u(d)$ be the constant as in the previous lemma. Then, 
\begin{align*}
\lim_{n \rightarrow \infty}(\var{L_{n}}-\frac{2d-1}{d^2}n)= \beta(d), 
\end{align*}
where $\beta(d)= -2u(d)+ \frac{(2d-1)(4d^{2}-6d+3)}{4d^{2}(d-1)^{2}}$. 
\end{prop}

\begin{proof}
Since $\var{D_n}=\EE\var(D_n|R_{n-1})+\var\EE(D_n|R_{n-1})\rightarrow \frac{(2d-1)(4d^{2}-6d+3)}{4d^{2}(d-1)^{2}}$, we have: 
\begin{align*}
\var{L_n}-\frac{2d-1}{d^2}n = 2\cov(S_n,D_n)+\var{D_n} \rightarrow \beta(d),
\end{align*}
where $\beta(d)= -2u(d)+ \frac{(2d-1)(4d^{2}-6d+3)}{4d^{2}(d-1)^{2}}$. 
\end{proof}

\bigskip

Combining the above estimates for $L_n$, we then have similar estimates for $T_n$: 

\begin{align*}
\EE T_n = \EE\EE(T_n|L_n) = \frac{2d-2}{2d-1}\EE{L_n}-\frac{2d-2}{2d-1}
\end{align*}
\begin{align*}
\var{T_n}&=\EE\var(T_n|L_n)+\var\EE(T_n|L_n) \\
&=\frac{2(d-1)}{(2d-1)^2}\EE{L_n}+\frac{4(d-1)^2}{(2d-1)^2}\var{L_n}-\frac{2(d-1)}{(2d-1)^2}, 
\end{align*}

which gives the following proposition: 

\bigskip

\begin{prop} \label{variance for number of turns}
Let $\beta(d)$ be the error term in $\var L_n$ as above. Then, 
\begin{align*}
\lim_{n \rightarrow \infty}(\EE T_n - \frac{2(d-1)^2}{d(2d-1)}n)=-\frac{2d^{2}-4d+1}{d(2d-1)},
\end{align*}
and
\begin{align*}
\lim_{n \rightarrow \infty}(\var{T_n} - \frac{2(d-1)^{2}(5d-2)}{d^{2}(2d-1)^{2}}n)= \frac{4(d-1)^2}{(2d-1)^2}{\beta(d)}-\frac{2d^{2}-4d+1}{d(2d-1)^{2}}. \\
\end{align*}
\end{prop}

\bigskip

\section{Large deviations}

The goal of this section is to prove the following large deviations theorem for $\frac{T_n}{n}$. 

\bigskip

\begin{thm} \label{large deviations for turns}
The sequence of the laws for the random variables $\{\frac{T_n}{n}\}_{n \geq 1}$ satisfy the large deviations principle with rate function
\begin{align*}
I(x) = \sup_{\theta} [\theta x - \log h(\theta)], 
\end{align*}
where
\begin{align*} 
h(\theta) = \left \{
\begin{array}{rl}
& \frac{1}{2d} [2(d-1)e^{\theta} + \frac{2d-1}{1 + 2(d-1)e^{\theta}} + 1], \theta \geq \log \frac{\sqrt{2d-1}-1}{2(d-1)} \\
&\frac{\sqrt{2d-1}}{d}, \qquad \qquad \qquad \qquad \qquad                 \theta < \log \frac{\sqrt{2d-1}-1}{2(d-1)}
\end{array} \right., 
\end{align*}
\end{thm}

\bigskip

We postpone the proof of this theorem to the end of the section. In light of Gartner-Ellis theorem (see \cite{Dembo} section 2.3), it suffices to show that
\begin{align*}
\lim_{n \rightarrow +\infty} (\EE e^{\theta T_n})^{\frac{1}{n}} = h(\theta)
\end{align*}
for every $\theta \in \RR$, and the limit $h$ is essentially smooth. 

Note that $T_{n} | L_{n}$ has a binomial distribution with parameter $(\frac{2d-2}{2d-1}, L_{n}-1)$ for $L_{n} \geq 1$, and $T_{n} = 0$ if $L_{n} = 0$, we have
\begin{align} \label{moment generating function for turns}
\EE e^{\theta T_{n}} = \PP(L_{n} = 0) + \PP(L_{n} > 0) \EE w(\theta)^{L_{n} -1}, 
\end{align}
where $w(\theta) = \frac{2d-2}{2d-1}e^{\theta} + 1$. The first term is bounded by
\begin{align*}
\PP(L_{n} = 0) &\leq \sum_{k=0}^{n} \PP(S_{n} = -k) \\
&\leq C (\frac{\sqrt{2d-1}}{d})^{n}, 
\end{align*}
and we need to estimate $\EE w(\theta)^{L_{n}}$ for large $n$.

In the context below, we regard $w$ to be a positive real number independent of $\theta$, and study the asymptotics of $(\EE w^{L_{n}})^{\frac{1}{n}}$ as $n \rightarrow +\infty$. 

\bigskip

\begin{prop} \label{the limit for larger w}
If $w \geq 1$, then we have
\begin{align*}
\lim_{n \rightarrow +\infty} (\EE w^{L_{n}})^{\frac{1}{n}} = \frac{2d-1}{2d}w + \frac{1}{2d} \cdot \frac{1}{w}. 
\end{align*}
\end{prop}
\begin{proof}
We compare the difference between $\EE w^{L_{n+1}}$ and $\EE w^{L_n}$: 
\begin{align*}
\EE w^{L_{n+1}} &= \EE 1_{\{L_n=0\}}w^{L_{n+1}} + \EE 1_{\{L_n>0\}}w^{L_{n+1}}\\
&= w \PP(L_n=0) + \EE 1_{\{L_n>0\}} w^{L_{n}+X_{n+1}}\\
&= w \PP(L_n=0) + \EE w^{X_{n+1}} \EE 1_{\{L_n>0\}} w^{L_n}\\
&= w \PP(L_n=0) + \EE w^{X_{n+1}} \EE w^{L_n} - \EE w^{X_{n+1}} \PP(L_n=0)\\&= (\frac{2d-1}{2d}w+\frac{1}{2d}\frac{1}{w})\EE w^{L_n} + \frac{1}{2d}(w-\frac{1}{w})\PP(L_n=0)
\end{align*}

Let $x_n=\EE w^{L_{n}}, a=\frac{2d-1}{2d}w+\frac{1}{2d}\frac{1}{w}, b=\frac{1}{2d}(w-\frac{1}{w}), p_n=\PP(L_n=0)$, we have the following recursive relation: 
\begin{align*}
x_n=a x_{n-1}+bp_{n-1}
\end{align*}

Since $x_1=w$, adding them up yields: 
\begin{align*}
x_n=a^{n-1}w+b(a^{n-2}p_{1}+a^{n-3}p_{2}+\cdots+ap_{n-2}+p_{n-1})
\end{align*}

For $w \geq 1$, we have $b \geq 0$. In this case, since $a^{n-1} \leq x_n \leq na^{n-1}$ from the expression above, we get: 
\begin{align*}
\lim_{n \rightarrow \infty}{(x_n)^{\frac{1}{n}}}= a=\frac{2d-1}{2d}w+\frac{1}{2d}\frac{1}{w}, 
\end{align*}
thus proving the proposition. 
\end{proof}

The situation for $w \in (0,1)$ is more involved. We prove it based on comparison with $(\EE w^{S_n})^{\frac{1}{n}}$. Note that $S_n$ is a sum of i.i.d., so by Cramer's theorem, it satisfies large deviations principle with rate function
\begin{align*}
J(x) = \sup_{\theta} [\theta x - \log(\frac{2d-1}{2d} e^{\theta} + \frac{1}{2d} e^{-\theta})]. 
\end{align*}

\bigskip

\begin{lem}
For any $w>0$, the equation
\begin{align} \label{equation for the position variable}
w^{\alpha} e^{-J(\alpha)} = \frac{2d-1}{2d}w + \frac{1}{2d} \cdot \frac{1}{w}
\end{align}
has a unique solution at $\alpha^{*} = \frac{(2d-1)w^{2} - 1}{(2d-1)w^{2} + 1}$. Furthermore, $\alpha^{*}$ is the global maximizer for
\begin{align*}
f_{w}(\alpha) = w^{\alpha}e^{-J(\alpha)}. 
\end{align*}
\end{lem}
\begin{proof}
We first give an expression of $J$ in terms of $\alpha$ only. It is clear that $J(\alpha) = +\infty$ for $|\alpha| > 1$. For $\alpha \in [-1,1]$, the maximizer $\theta^{*}$ is
\begin{align*}
\theta^{*}(\alpha) = \frac{1}{2}[\log \frac{1+\alpha}{1-\alpha} - \log(2d-1)], \alpha \in [-1,1], 
\end{align*}
passing to the limit $\pm \infty$ for $\alpha = \pm 1$. Substituting into $J$, we have
\begin{align*}
J(\alpha) = \frac{1}{2} \alpha [\log \frac{1+\alpha}{1-\alpha} - \log(2d-1)] - \log \frac{\sqrt{2d-1}}{2d} - \log(\sqrt{\frac{1+\alpha}{1-\alpha}} + \sqrt{\frac{1-\alpha}{1+\alpha}})
\end{align*}
for $\alpha \in [-1,1]$. Differentiating with respect to $\alpha$, we obtain
\begin{align} \label{derivative of J}
J'(\alpha) = \frac{1}{2}[\log \frac{1+\alpha}{1-\alpha} - \log(2d-1)]
\end{align}
for $\alpha \in (-1,1)$. Note that
\begin{align*}
f_{w}'(\alpha) = \frac{d}{d \alpha}(w^{\alpha} e^{-J(\alpha)}) = w^{\alpha} e^{- J(\alpha)} (\log w - J'(\alpha)), 
\end{align*}
and since $J$ is convex, $f_{w}$ has the global maximizer $\alpha^{*}$ satisfying
\begin{align*}
J'(\alpha^{*}) = \log w. 
\end{align*}
By \eqref{derivative of J}, solving the above first order condition yieds
\begin{align*}
\alpha^{*} = \frac{(2d-1)w^{2} - 1}{(2d-1)w^{2} + 1}, 
\end{align*}
and thus
\begin{align*}
f_{w}(\alpha^{*}) = \frac{2d-1}{2d}w + \frac{1}{2d} \cdot \frac{1}{w}. 
\end{align*}
Since $\alpha^{*}$ is the global maximier of $f_{w}$, we conclude that equation \eqref{equation for the position variable} has a unique solution at $\alpha^{*}$. 
\end{proof}

\bigskip

\begin{prop}
Let $\alpha^{*} = \alpha^{*}(w) = \frac{(2d-1)w^{2} - 1}{(2d-1)w^{2} + 1}$ be as above, then
\begin{align*}
\lim_{\epsilon \downarrow 0} \lim_{n \rightarrow +\infty} (\EE w^{S_n} 1_{\{\frac{S_n}{n} \in (\alpha^{*}-\epsilon, \alpha^{*}+\epsilon)\}}) = \frac{2d-1}{2d}w + \frac{1}{2d} \cdot \frac{1}{w}. 
\end{align*}
\end{prop}

\bigskip

This proposition shows that the major contribution for $\EE w^{S_n}$ are from the $S_n$'s with values near $\alpha^{*} n$.

\bigskip

\begin{lem}
(a) Let $\alpha > 0$. $\forall \epsilon \in (0,\alpha)$, $\forall \delta > 0$, $\exists N = N(\alpha, \epsilon, \delta)$ such that
\begin{align*}
\PP(L_{n} - S_{n} \leq \delta n | \frac{S_n}{n} \in (\alpha - \epsilon, \alpha + \epsilon)) \geq \frac{1}{2} 
\end{align*}
for all $n \geq N$. \\
(b) $\forall \epsilon, \delta > 0$, $\exists N = N(\epsilon, \delta)$ such that
\begin{align*}
\PP(L_{n} \leq \delta n | S_{n} \leq - \epsilon n) \geq \frac{1}{2}
\end{align*}
for all $n \geq N$. 
\end{lem}
\begin{proof}
We prove part(a), and the proof for part (b) is similar. Observe that 
\begin{align*}
D_{n} = L_{n} - S_{n} = 2 |\min_{0 \leq k \leq n} S_{k}|,
\end{align*}
we first consider the quantity $\PP(\min_{1 \leq k \leq n} S_k \geq - \delta n | \frac{S_n}{n} = \alpha n)$, where without loss of generality, we have assumed $\frac{1+\alpha}{2} n$ is an integer, and have replaced $\frac{\delta}{2}$ by $\delta$. Once conditioned on the event $\{\frac{S_n}{n} = \alpha n\}$, all possible paths contain $\begin{pmatrix} n\\ \frac{1+\alpha}{2}n \end{pmatrix}$ positive movements and $\begin{pmatrix} n\\ \frac{1-\alpha}{2} n \end{pmatrix}$ negative movements. Since all these paths have the same (conditional) weight, the quantity $\PP(\min_{1 \leq k \leq n} S_k \geq -\delta n | \frac{S_n}{n} = \alpha n)$ is independent of $d$. Thus, we may assume $d=1$, where all paths are the trajectories of the (conditional) simple symmetric random walk. That is, 
\begin{align*}
\PP(\min_{0 \leq k \leq n}S_{k} < -\delta n | S_{n} = \alpha n) = \PP(\min_{0 \leq k \leq n}\tilde{S}_{k} < -\delta n | \tilde{S}_{n} = \alpha n), 
\end{align*}
where $\tilde{S}_{n}$ is a one-dimensional simple symmetric random walk. By reflection principle, we have
\begin{align*}
\PP(\min_{0 \leq k \leq n}\tilde{S}_{k} < - \delta n, \tilde{S}_{n} = \alpha n) = \PP(\tilde{S}_{n} = - \lfl (\alpha + 2 \delta) \rfl n ). 
\end{align*}

Using Stirling's approximation, we estimate the ratio
\begin{align*}
\frac{\PP(\tilde{S}_{n} = -\lfl (\alpha + 2 \delta)n) \rfl}{\PP(\tilde{S}_{n} = \alpha n)} &= \begin{pmatrix} n\\ \frac{1 - \alpha - 2\delta}{2}n \end{pmatrix} / \begin{pmatrix} n\\ \frac{1+\alpha}{2}n \end{pmatrix} \\
&= (\frac{1-\alpha}{2}n)! (\frac{1+\alpha}{2}n)! / [(\frac{1 - \alpha - 2 \delta }{2}n)!  (\frac{1 + \alpha + 2\delta}{2}n)!] \\
&\approx [\frac{(1-\alpha)^{1-\alpha}(1+\alpha)^{1+\alpha}}{(1 - \alpha - 2\delta)^{1 - \alpha - 2\delta} (1 + \alpha + 2\delta)^{1 + \alpha + 2\delta}}]^{\frac{n}{2}}. 
\end{align*}

It is straightforward to check that $(1-x)\log(1-x) + (1+x)\log(1+x)$ is increasing in $(0,1)$, thus the quantity in the bracket in the last line is strictly less than one, and the ratio decays exponentially. So, we have
\begin{align*}
&\PP(\min_{0 \leq k \leq n}S_{k} < -\delta n | \frac{S_n}{n} \in (\alpha - \epsilon, \alpha + \epsilon)) \\
&= \sum_{\beta \in (\alpha - \epsilon, \alpha + \epsilon)} \PP(\min_{1 \leq k \leq n}\tilde{S}_{k} < -\delta n, \frac{\tilde{S}_{n}}{n} \in (\alpha - \epsilon, \alpha + \epsilon)) / \PP(\frac{\tilde{S}_{n}}{n} \in (\alpha - \epsilon, \alpha + \epsilon)) \\
&\leq n \sup_{\beta \in (\alpha - \epsilon, \alpha + \epsilon)} \PP(\tilde{S}_{n} = -(\beta + 2 \delta) n ) / \PP(\frac{\tilde{S}_{n}}{n} \in (\alpha - \epsilon, \alpha + \epsilon)) \\
&\rightarrow 0, 
\end{align*}
and consequently
\begin{align*}
\PP(\min_{0 \leq k \leq n} \geq -\delta n | \frac{S_n}{n} \in (\alpha - \epsilon, \alpha + \epsilon)) \rightarrow 1 
\end{align*}
as $n \rightarrow +\infty$, thus proving (a). The proof for part (b) is essentially the same, and we omit it here.

\end{proof}

\bigskip

\begin{prop} \label{the limit for smaller w}
For $w \in (0,1)$, the limit $(\EE w^{L_n})^{\frac{1}{n}}$ exists as $n \rightarrow +\infty$, and we have
\begin{align*} 
\lim_{n \rightarrow +\infty} (\EE w^{L_n})^{\frac{1}{n}} = \left \{
\begin{array}{rl}
&\frac{2d-1}{2d} w + \frac{1}{2d} \cdot \frac{1}{w}, \qquad w \in (\frac{1}{\sqrt{2d-1}}, 1)\\
&\frac{\sqrt{2d-1}}{d}, \qquad \qquad \qquad w \in (0, \frac{1}{\sqrt{2d-1}}]
\end{array} \right., 
\end{align*}
\end{prop}
\begin{proof} 

\begin{enumerate}

\item $w \in (\frac{1}{\sqrt{2d-1}},1)$, $\alpha^{*} = \frac{(2d-1)w^{2}-1}{(2d-1)w^{2}+1} > 0$. 

In this case, we have
\begin{align*}
(\EE w ^{L_{n}})^{\frac{1}{n}} &\geq [\EE (w^{L_n} 1_{\{L_{n} - S_{n} \leq \delta n\}} | \frac{S_n}{n} \in (\alpha^{*} - \epsilon, \alpha^{*} + \epsilon)) \PP(\frac{S_n}{n} \in (\alpha^{*} - \epsilon, \alpha^{*} + \epsilon))]^{\frac{1}{n}} \\
&\geq w^{\alpha^{*} + \epsilon + \delta} \PP(\frac{S_n}{n} \in (\alpha^{*} - \epsilon, \alpha^{*} + \epsilon))^{\frac{1}{n}}  \PP(L_n - S_n \leq \delta n | \frac{S_n}{n} \in (\alpha^{*} - \epsilon, \alpha^{*} + \epsilon))^{\frac{1}{n}}. 
\end{align*}
Taking $n \rightarrow +\infty$ on both sides yields
\begin{align*}
\liminf_{n \rightarrow +\infty} (\EE w^{L_n})^{\frac{1}{n}} \geq w^{\alpha^{*} + \epsilon + \delta} \cdot e^{-I(\alpha^{*} + \epsilon)}. 
\end{align*}
Since $\epsilon$ and $\delta$ are arbitrary, we have
\begin{align*}
\liminf_{n \rightarrow +\infty} (\EE w^{L_{n}})^{\frac{1}{n}} \geq w^{\alpha^{*}} e^{-I(\alpha^{*})} = \frac{2d-1}{2d} w + \frac{1}{2d} \cdot \frac{1}{w}. 
\end{align*}
On the other hand, $\EE w^{L_n} \leq \EE w^{S_n}$ for all $w < 1$, so
\begin{align*}
\limsup_{n \rightarrow +\infty} (\EE w^{L_{n}})^{\frac{1}{n}} \leq \frac{2d-1}{2d} w + \frac{1}{2d} \cdot \frac{1}{w}. 
\end{align*}
Thus, we conclude that
\begin{align*}
\lim_{n \rightarrow +\infty} (\EE w^{L_{n}})^{\frac{1}{n}} = \frac{2d-1}{2d} w + \frac{1}{2d} \cdot \frac{1}{w}
\end{align*}
for $w \in (\frac{\sqrt{2d-1}}{d}, 1)$. 

\bigskip

\item $w \in (0, \frac{1}{\sqrt{2d-1}})$, and $\alpha^{*} < 0$. 

In this case, we have
\begin{align*}
(\EE w^{L_n})^{\frac{1}{n}} &\geq [\EE(w^{L_n} 1_{\{L_n \leq \delta n\}} | S_n \leq - \epsilon n) \PP(S_n \leq - \epsilon n)]^{\frac{1}{n}} \\
&\geq w^{\delta} \PP(L_n \leq \delta n | S_n \leq - \epsilon n)^{\frac{1}{n}} \PP(S_n \leq - \epsilon n)^{\frac{1}{n}}. 
\end{align*}
Again, sending $n \rightarrow +\infty$ and take $\epsilon, \delta$ arbitrarily small yields
\begin{align*}
\liminf_{n \rightarrow +\infty}(\EE w^{L_n})^{\frac{1}{n}} \geq e^{- I(0)} = \frac{\sqrt{2d-1}}{d}. 
\end{align*}
On the other hand, 
\begin{align*}
\lim_{w \downarrow \frac{1}{\sqrt{2d-1}}} \lim_{n \rightarrow +\infty} (\EE w^{L_n})^{\frac{1}{n}} = \frac{\sqrt{2d-1}}{d}, 
\end{align*}
by monotonocity, we have
\begin{align*}
\limsup_{n \rightarrow +\infty} (\EE w^{L_{n}})^{\frac{1}{n}} \leq \frac{\sqrt{2d-1}}{d}, 
\end{align*}
and thus
\begin{align*}
\limsup_{n \rightarrow +\infty} (\EE w^{L_{n}})^{\frac{1}{n}} = \frac{\sqrt{2d-1}}{d}
\end{align*}
for all $w \in (0, \frac{1}{\sqrt{2d-1}})$. 

\bigskip

\item $w = \frac{1}{\sqrt{2d-1}}$, $\alpha^{*} = 0$. 
Again, by monotonocity in $w$, we have
\begin{align*}
\lim_{n \rightarrow +\infty} (\EE \sqrt{2d-1}^{-{L_n}})^{\frac{1}{n}} = \frac{\sqrt{2d-1}}{d}. 
\end{align*}

\end{enumerate}
\end{proof}

\bigskip

The following corollary is an immediate consequence of Propositions \ref{the limit for larger w} and \ref{the limit for smaller w}.

\bigskip

\begin{cor}
The laws for the random variables $\{\frac{L_n}{n}\}_{n \geq 1}$ satisfy the large deviations principle with rate function
\begin{align*}
I_{L}(x) = \sup_{\theta} [\theta x - \log h_{L}(\theta)], 
\end{align*}
where
\begin{align*}
h_{L}(\theta) = \left \{
\begin{array}{rl}
& \frac{2d-1}{2d}e^{\theta} + \frac{1}{2d}e^{- \theta}, \qquad   \theta \geq - \frac{1}{2} \log(2d-1) \\
&\frac{\sqrt{2d-1}}{d}, \qquad \qquad \qquad                     \theta < \log \frac{\sqrt{2d-1}-1}{2(d-1)}
\end{array} \right., 
\end{align*}
\end{cor}

\bigskip

Now we are in a position to prove Theorem \ref{large deviations for turns}.

\bigskip

\begin{flushleft}
\textbf{Proof of Theorem \ref{large deviations for turns}. }
\end{flushleft}

\begin{proof}
Recall that
\begin{align*}
\EE e^{\theta T_{n}} = \PP(L_{n} = 0) + \PP(L_{n} \geq 1) \EE w(\theta)^{L_{n} - 1}, 
\end{align*}
where we have
\begin{align*}
\PP(L_{n} = 0) \leq C (\frac{\sqrt{2d-1}}{d})^{n}
\end{align*}
for all $n$. On the other hand, Propositions \ref{the limit for larger w} and \ref{the limit for smaller w} imply that
\begin{align*}
\lim_{n \rightarrow +\infty} (\EE w^{L_{n}})^{\frac{1}{n}} \geq \frac{\sqrt{2d-1}}{d}
\end{align*}
for all $w > 0$. Thus, we see that
\begin{align*}
\lim_{n \rightarrow +\infty} (\EE e^{\theta T_{n}})^{\frac{1}{n}} = \lim_{n \rightarrow +\infty} (\EE w(\theta)^{L_n})^{\frac{1}{n}} = h(\theta), 
\end{align*}
where $h$ is defined in the theorem. It is also straigtforward to check that $h$ is essentially smooth. Thus, Gartner-Ellis theorem implies that the laws for $\{\frac{T_n}{n}\}$ satisfies the large deviations principle with rate function
\begin{align*}
I(x) = \sup_{\theta} [\theta x - \log h(\theta)]. 
\end{align*}
\end{proof}

\bigskip

We end this section with two remarks. 

\bigskip

\begin{rmk}
As an alternative to Gartner-Ellis theorem, one can compute the rate function for $\frac{T_n}{n}$ directly as follows: 
\begin{align*}
\PP(\frac{T_n}{n} \in (\alpha - \epsilon, \alpha + \epsilon)) = \sum_{\beta} \PP(\frac{T_n}{n} \in (\alpha - \epsilon, \alpha + \epsilon) | \frac{L_n}{n} \in (\beta - \delta, \beta + \delta)) \cdot \PP(\frac{L_n}{n} \in (\beta - \delta, \beta + \delta))
\end{align*}
where the sum is taken over appropriate $\beta$'s $\in (0,1)$. In each product, the first probability is known as $T_n | L_n$ has a binomial distribution, while the second probability can be aysmptotically computed from the LDP for $L_n$. Finally, sending $\delta \rightarrow 0$, one can get the rate function for $\{\frac{T_n}{n}\}$. 
\end{rmk}

\bigskip

\begin{rmk}
We give a heuristic explaination why the rate function of $\{\frac{L_n}{n}\}$ takes this form. Let $w = e^{\eta}$, and consider the Laplace transform
\begin{align*}
(\EE w^{L_n})^{\frac{1}{n}} = (\EE e^{\eta L_n})^{\frac{1}{n}}. 
\end{align*}
For $w > w^{*} = \frac{1}{\sqrt{2d-1}}$, the limit exists and is equal to $\frac{2d-1}{2d} e^{\eta} + \frac{1}{2d} e^{- \eta}$, 
where $w = e^{\eta}$. Call the limit of the right hand side $\Lambda(\eta)$, if it exists. One sees that $\eta^{*} = \log w^{*}$ is the minimizer of $\Lambda$, and thus
\begin{align*}
\Lambda(\eta) \geq \frac{\sqrt{2d-1}}{d}
\end{align*}
for all $\eta < \eta^{*}$. On the other hand, if $\Lambda(\eta)$ exists, it must be convex. Because $L_n \geq 0$, it also must not (strictly) increase when $\eta$ becomes smaller. So the convexity in $\eta$ forces the graph to the left of $\eta^{*} = \log w^{*}$ to be flat, and is thus equal to $\frac{\sqrt{2d-1}}{d}$. 
\end{rmk}

\bigskip

\section{Central limit theorem and the invariance principle}

In this section, we derive asymptotic distribution of $T_n$ for large $n$, including central limit theorem and the invariance principle. Under the natural coupling $L_n = S_n + D_n$, $D_n \leq 2 R_n$, where $R_n$ is (almost) surely bounded by a geometric random variable $R$. As we will be frequently using this property, we state it as a lemma below.  

\bigskip

\begin{lem}
Let $L_n = S_n + D_n$, and $R_n$ be defined as above. Then, $D_n \leq w R_n$, and $R = \sup_{n}R_n$ has a geometric distribution with $\PP(R_n = k) = \frac{2d-2}{2d-1} \cdot (\frac{1}{2d-1})^{\frac{1}{n}}$. 
\end{lem}

\bigskip

\begin{thm}
Under the above assumptions, $\frac{1}{\sqrt{n}}(L_{n}-\frac{d-1}{d}n)$ converges in distribution to $N(0, \frac{2d-1}{d^2})$, and $\frac{1}{\sqrt{n}}(T_{n}-\frac{2(d-1)^2}{d(2d-1)}n)$ converges in distribution to $N(0, \frac{2(d-1)^{2}(5d-2)}{d^{2}(2d-1)^{2}})$. 
\end{thm}
\begin{proof}
For any $x \in \RR$, 
\begin{align*}
F_{n}(x)&=\PP(\frac{1}{\sqrt{n}}(L_{n}-\frac{2d-1}{d^2}n) \leq x)\\
&=\PP(\frac{1}{\sqrt{n}}(S_{n}-\frac{2d-1}{d^2}n)+\frac{D_n}{\sqrt{n}} \leq x) \\
& \rightarrow F(x),
\end{align*}
where $F$ is the distribution function for $N(0, \frac{2d-1}{d^2})$. The convergence in the last line follows from the fact that $\frac{D_n}{\sqrt{n}} \rightarrow 0$ almost surely, as $R = \sup_{n}R_n$ is almost surely finite. \\

Now we compute the moment generating function for $\frac{1}{\sqrt{n}}(T_{n}-\frac{2(d-1)^2}{d(2d-1)}n)$. For simplicity, let $\mu = \frac{2(d-1)^2}{d(2d-1)}$ and $p=\frac{2d-2}{2d-1}$. 
\begin{align*}
\EE e^{\frac{\theta}{\sqrt{n}}(T_{n}-\mu n)}&= e^{-\theta \mu \sqrt{n}}\EE e^{\frac{\theta}{\sqrt{n}}T_{n}}\\
&= e^{-\theta \mu \sqrt{n}} \EE \EE (e^{\frac{\theta}{\sqrt{n}}T_{n}}|L_n)\\
&= e^{-\theta \mu \sqrt{n}} \EE (1-p+pe^{\frac{\theta}{\sqrt{n}}})^{L_{n}-1}\\
&= e^{-\theta \mu \sqrt{n}} \EE e^{(L_{n}-1) \log(1+\frac{p\theta}{\sqrt{n}}+\frac{p{\theta}^2}{2n}+o(\frac{1}{n}))}\\
&= e^{-\theta \mu \sqrt{n}} \EE e^{(L_{n}-1)[\frac{p\theta}{\sqrt{n}}+\frac{1}{2n}(p{\theta}^2-p^{2}{\theta}^{2})+o(\frac{1}{n})]}\\
&= \EE e^{\frac{p\theta}{\sqrt{n}}(L_{n}-\frac{\mu}{p}n)+\frac{1}{2}{\theta}^{2}p(1-p)\frac{L_n}{n}+o(1)}\\
& \rightarrow e^{\frac{(d-1)^{2}(5d-2)}{d^{2}(2d-1)^{2}}{\theta}^{2}}
\end{align*}
The convergence in the last line follows from the law of large numbers and central limit theorem for $L_n$, and dominated convergence. This implies that $\frac{1}{\sqrt{n}}(T_{n}-\frac{2(d-1)^2}{d(2d-1)}n)$ converges in distribution to $N(0, \frac{2(d-1)^{2}(5d-2)}{d^{2}(2d-1)^{2}})$. \\
\end{proof}

\bigskip

Since $L_n$ behaves very much like $S_n$, one expects that under proper scaling, it converges to the standard Brownian motion. 

\bigskip

\begin{prop}
Let $(\Omega,\PP, \cF)$ be a probability space affording the discrete time process $\{S_i, i \in \NN\}$. For each $n \in \NN$, define a $C^0([0,1])$-valued random variable $\{W^{(n)}_t: t\in [0,1]\}$ by 
\begin{align*}
W^{(n)}_t = [L_{tn}- \frac{d-1}{d} tn]/\sqrt{n \frac{2d-1}{d^2}}
\end{align*}
for $t \in \frac{1}{n} [n]$, and linearly interpolated for other values of $t$. Then the sequence converges in law to the standard one dimensional Brownian motion on $[0,1]$ as $n \to \infty$.
\end{prop}

\begin{proof}
This is plain in light of Lemma 3.4, as $L_{[tn]}-S_{[tn]}$ is bounded by a geometric random variable, uniformly in $t$. 
\end{proof}

\bigskip

More difficult is the invariance principle for $T_n$, the number of turns at time $n$. We prove it in the next theorem. Let $\sigma ^2 = \frac{2(d-1)^{2}(5d-2)}{d^{2}(2d-1)^{2}}$, then, 

\bigskip

\begin{thm}
For each $n$,define a $C^0([0,1])$-valued random variable $\{W^{(n)}_t: t\in [0,1]\}$ by 
\begin{align} \label{W for T}
W^{(n)}_t = \frac{1}{\sigma \sqrt{n}}[T_{tn}- \frac{2(d-1)^2}{d(2d-1)} tn]
\end{align}
for $t \in \frac{1}{n} [n]$, and linearly interpolated for other values of $t$. Then the sequence converges in law to the standard one dimensional Brownian motion on $[0,1]$ as $n \to \infty$.
\end{thm}
\begin{proof}

Recall for a sequence of processes to converge to Brownian motion, it suffices to check that their finite dimensional joint distributions converge to that of a Brownian motion and Prohorov tightness criterion (see Theorem 16.5 of \cite{Kallenberg}).  We prove them in the two lemmas below. 

\bigskip

\begin{lem}
Let $W^{(n)}_{t}$ be defined as above. Then, 
\begin{align*}
(W^{(n)}_{t_1}, W^{(n)}_{t_2}, \cdots, W^{(n)}_{t_k}) \rightarrow (B_{t_1}, B_{t_2}, \cdots, B_{t_k}), 
\end{align*}
where $B_t$ is a standard one-dimensional Brownian motion starting at $0$. 
\end{lem}
\begin{proof}
It suffices to show that $\cL(W^{(n)}_t - W^{(n)}_s| W^{(n)}_r: r \le s)$ is asymptotically $N(0,t-s)$ and independent of $\{W^{(n)}_r: r \le s\}$. This is a generalization of the central limit theorem for $T_n$, which corresponds to the case $s = 0$, and also suggests that for general $s$, we can compare with the case $s=0$. We couple $W^{(n)}_{t}-W^{(n)}_{s}$ with two extremely cases, as considered below. 

Let $T_{(\cdot)}$ be the standard process of number of turns, and define
\begin{align*}
U_{(m,n)}:= T|_{[m,n]}, 
\end{align*}
that is, the number of turns in the segment of the walk during the time interval $[m,n]$. Then, one immediately sees that
\begin{align*}
T_{\lfl tn \rfl} - T_{\lfl sn \rfl} \leq U_{[\lfl sn \rfl, \lfl tn \rfl]}, 
\end{align*}
as the former may cancel turns created before time $sn$. 

On the other hand, we have the reversed inequality that
\begin{align*}
T_{\lfl tn \rfl} - T_{\lfl sn \rfl} \geq U_{[\lfl sn \rfl, \lfl tn \rfl]} - 2 |\min_{1 \leq k \leq \lfl (t-s)n \rfl} S_{k}|, 
\end{align*}
and it suffices to estimate the behavior of $|\min_{1 \leq k \leq n}S_{k}|$. Since
\begin{align*}
\PP(|\min_{1 \leq k \leq n}| = M ) \leq \sum_{k=0}^{\lfl (t-s)n \rfl} \PP(S_{k} = -M) \leq C (\frac{\sqrt{2d-1}}{d})^{M}, 
\end{align*}
we see that it is bounded by a geometric random variable, and thus
\begin{align*}
\frac{1}{\sigma n} |\min_{1 \leq k \leq (t-s)n} S_{k}| \rightarrow 0
\end{align*}
in probability. Note that $U_{[m,n]}$ has the same distribution as $T_{n-m}$, therefore $\mathcal{L}(W_{t}^{n} - W_{s}^{n} | W^{n}(0,s))$ converges to $N(0,t-s)$, and is clearly independent of $W^{n}(0,2)$. By induction, we also get the asymptitically independent increment property for the sequence $W^{n}$. 

\end{proof}

\bigskip

\begin{lem}
$W^{(n)}_{t}$ satisfies Prohorov's tightness condition, i.e., 
\begin{align*}
\lim_{h \to 0} \limsup_{n \to \infty} \EE [ \sup_{|t-s| \le h} |W^{(n)}_t - W^{(n)}_s| \wedge 1] = 0
\end{align*}
\end{lem}

\begin{proof}
Let $\bar{T}_n := T_n - \frac{2(d-1)^2}{d(2d-1)}n$ be the centered number of turns at step $n$. The idea of the proof is similar to Ottaviani's maximal inequality for random walk on $\RR$ (see Lemma 14.8 of \cite{Kallenberg}). 

Fix $\epsilon > 0$. Let $t \in [0,1)$ and $ h \in (0,1-t)$.  Define $\tau := \min\{ k \in (tn,(t+h)n]: |\bar{T}_k - \bar{T}_{[tn]}| \ge 2 \epsilon \sigma \sqrt{n}\}$. Then, 
\begin{align*}
\PP(| \bar{T}_{\lfl (t+h)n \rfl} - \bar{T}_{\lfl tn \rfl} | > \epsilon \sigma \sqrt{n}) &\geq \PP(\tau \leq (t+h)n, | \bar{T}_{\lfl (t+h)n \rfl} - \bar{T}_{\tau} | \leq \epsilon \sigma \sqrt{n}) \\
&= \sum_{k=\lfl tn \rfl +1}^{\lfl (t+h)n \rfl} \PP(\tau = k) \PP( | \bar{T}_{\lfl (t+h)n \rfl} - \bar{T}_{k} | \leq \epsilon \sigma \sqrt{n} | \tau = k ) \\
&\geq \PP(\tau \leq n) \min_{ k \in (tn,(t+h)n]} \PP( | \bar{T}_{\lfl (t+h)n \rfl} - \bar{T}_{k} | \leq \epsilon \sigma \sqrt{n} | \tau = k ). 
\end{align*}

We want to bound $\PP(\tau \leq n) = \PP(\max_{k} |\bar{T}_{k} - \bar{T}_{\lfl tn \rfl}| > 2 \epsilon \sigma \sqrt{n})$. First note that when $hn$ is large enough, $\bar{T}_{\lfl (t+h)n \rfl} - \bar{T}_{\lfl tn \rfl}$ behaves like a Gaussian with mean $0$ and variance $h \sigma^{2} n$. So, there exists $N(\epsilon)$ such that for all $n > \frac{N(\epsilon)}{h}$, we have
\begin{align} \label{Gaussian decay}
\PP(| \bar{T}_{\lfl (t+h)n \rfl} - \bar{T}_{\lfl tn \rfl} | > \epsilon \sigma \sqrt{n}) \leq \exp{( - \epsilon^{2} / 2h)}. 
\end{align}
On the other hand, $|\bar{T}_{\lfl (t+h)n) \rfl} - \bar{T}_{k}| \leq U_{(k,(t+h)n)} + M_{\lfl hn \rfl}$. According to Proposition \eqref{variance for number of turns}, the variance of the former term on the right hand side is to $(t+h)n-k$ with an error uniformly bounded in $t,h,k$ and $n$. The second term is dominated by a geometric random variable, and thus has a finite variance. So, by Chebyshev's inequality, we have
\begin{align*}
\PP(|\bar{T}_{\lfl (t+h)n \rfl} - \bar{T}_{k}| \leq \epsilon \sigma \sqrt{n} | \tau = k) \geq 1 - \frac{2[(t+h)n-k+C]}{n}. 
\end{align*}
Since $k$ takes values between $tn$ and $(t+h)n$, the minimum bound is achieved at $k=\lfl tn \rfl +1$, so we get
\begin{align*}
\min_{ k \in (tn,(t+h)n]} \PP( | \bar{T}_{\lfl (t+h)n \rfl} - \bar{T}_{k} | \leq \epsilon \sigma \sqrt{n} | \tau = k ) \geq 1 - \frac{2h}{\epsilon^{2}} - \frac{C}{n}, 
\end{align*}
where $C$ is independent of $t,h,k$ and $n$. The last inequality is valid as long as the right hand side is positive, which requires $h$ to take small values and $n$ to take large values. These values depend on $\epsilon$ only. The last bound, together with the bound \eqref{Gaussian decay}, imply
\begin{align} \label{Ottaviani}
\PP(\max_{k \in (tn, (t+hn)]} | \bar{T}_{k} - \bar{T}_{tn} | > 2 \epsilon \sigma \sqrt{n} ) \leq (1 - \frac{2h}{\epsilon^{2}} - \frac{C}{n})^{-1} \exp{( - \epsilon^{2} / 2h)}. 
\end{align}
Since for $n > \frac{16}{\epsilon^{2} \sigma^{2}}$ we have
\begin{align*}
\PP(\sup_{\delta \in (0,h)} | W_{t+\delta}^{(n)} - W_{t}^{(n)} | > \epsilon) \leq \PP(\sup_{k \in (tn,(t+h)n)} |\bar{T}_{k} - \bar{T}_{\lfl tn \rfl}| > \frac{\epsilon}{2} \sigma \sqrt{n}), 
\end{align*}
and the bound on the right hand side of \eqref{Ottaviani} is independent of $t$, one gets
\begin{align*}
\sup_{t \in (0,1)} \PP(\sup_{\delta \in (0,h)} | W_{t+\delta}^{(n)} - W_{t}^{(n)} | > \epsilon) < (1 - \frac{2h}{\epsilon^{2}} - \frac{C}{n})^{-1} \exp{( - \epsilon^{2} / 32h)}
\end{align*}
for all $n > \max \{ \frac{N(\epsilon)}{h}, \frac{16}{\epsilon^{2} \sigma^{2}}\}$. Now divide the interval $(0,1)$ into $\lfl \frac{1}{h} \rfl + 1$ subintervals, each with length at most $h$. Then, $|t-s| < h$ implies that either $s,t$ are in the same subinterval, or they are in two adjacent ones. This observation gives
\begin{align} \label{essential maximal inequality}
\PP(\sup_{|t-s|<h} | W_{t}^{(n)} - W_{s}^{(n)} | > \epsilon) < (\frac{2}{h} + 2) (1 - \frac{2h}{\epsilon^{2}} - \frac{C}{n})^{-1} \exp{( - \epsilon^{2} / 32h)}
\end{align}
for all $n > \max \{ N(\epsilon), \frac{16}{\epsilon^{2} \sigma^{2}}\}$. Since
\begin{align*}
\EE [\sup_{|t-s|<h} | W_{t}^{(n)} - W_{s}^{(n)}| \wedge 1] \leq \epsilon + \PP(\sup_{|t-s|<h} | W_{t}^{(n)} - W_{s}^{(n)} | > \epsilon), 
\end{align*}
the maximal inequality \eqref{essential maximal inequality} quickly gives
\begin{align*}
\lim_{h \to 0} \limsup_{n \to \infty} \EE [ \sup_{|t-s| \le h} |W^{(n)}_t - W^{(n)}_s| \wedge 1] \leq \epsilon, 
\end{align*}
which implies Prohov's tightness condition since $\epsilon$ is arbritary. 

\end{proof}

\bigskip

Combining the above two lemmas, we prove the invariance principle. 

\end{proof}

\end{document}